\documentclass{article}
\usepackage{amssymb,amsmath}
\usepackage{graphicx}

\def\qed{\hfill {\hbox{${\vcenter{\vbox{               
   \hrule height 0.4pt\hbox{\vrule width 0.4pt height 6pt
   \kern5pt\vrule width 0.4pt}\hrule height 0.4pt}}}$}}}

\newtheorem{theorem}{Theorem}
\newtheorem{definition}{Definition}

\newtheorem{proposition}[theorem]{Proposition}

\newtheorem{example}{Example}

\newtheorem{remark}{Remark}

\newenvironment{proof}[1][Proof]{\smallskip\noindent{\bf #1.}\quad}%
{\qed\par\medskip}

\date{}

\title{\Large \textbf{Birack Dynamical Cocycles and Homomorphism Invariants}}

\author{Sam Nelson\footnote{Email: knots@esotericka.org}\and
Emily Watterberg\footnote{Email:cewatterberg@comcast.net }}

\begin{document}
\maketitle

\begin{abstract} Biracks are algebraic structures related to knots and links.
We define a new enhancement of the birack counting invariant for 
oriented classical and virtual knots and links via algebraic structures 
called birack dynamical cocycles. The new invariants can also be understood 
in terms of partitions of the set of birack labelings of a link diagram 
determined by a homomorphism $p:X\to Y$ between finite labeling biracks.
We provide examples to show that the new invariant is stronger than
the unenhanced birack counting invariant and examine connections
with other knot and link invariants.
\end{abstract}

\medskip

\quad
\parbox{5.5in}{
\textsc{Keywords:} biracks, dynamical cocycles, birack homomorphims, 
enhancements of counting invariants
\smallskip

\textsc{2010 MSC:} 57M27, 57M25
}

\section{\large\textbf{Introduction}}\label{I}

Biracks were first introduced in \cite{FRS} as an algebraic structure with
axioms motivated by the framed Reidemeister moves. Biquandles, a special case 
of biracks, were developed in more detail in \cite{KR} and in later work such 
as \cite{FJK}. In \cite{N2} the \textit{integral birack counting invariant} 
$\Phi_X^{\mathbb{Z}}$, an integer-valued invariant of classical and 
virtual knots and links, was defined using labelings of knot and link diagrams
by finite biracks. More recent works such as \cite{BN} have defined 
\textit{enhancements} of the integral counting invariant, new invariants 
which are generally stronger but specialize to $\Phi_X^{\mathbb{Z}}$.

In this paper we define a new enhancement of the integral birack counting 
invariant using an algebraic structure called a \textit{birack dynamical 
cocycle}, analogous to rack dynamical cocycles introduced in \cite{AG}
and applied to enhancements in \cite{CNS}. We then reformulate and generalize
the new invariant in terms of birack homomorphisms. The paper is 
organized as follows. In Section \ref{B} we review the basics of biracks
and the birack counting invariant. In Section \ref{BP} we define the birack
dynamical cocycle invariant and discuss relationships with 
previously studied invariants. In Section \ref{CA} we collect
some computations and applications of the new invariant, and in Section 
\ref{Q} we finish with some open questions for future work.

\section{\large\textbf{Biracks and the Counting Invariant}}\label{B}

Recall that a \textit{framed link} can be defined combinatorially
as an equivalence class of link diagrams (projections of unions of 
simple closed curves in $\mathbb{R}^3$ onto a plane with breaks to 
indicate crossing information) under the equivalence relation
generated by the \textit{framed Reidemeister moves}:
\[\includegraphics{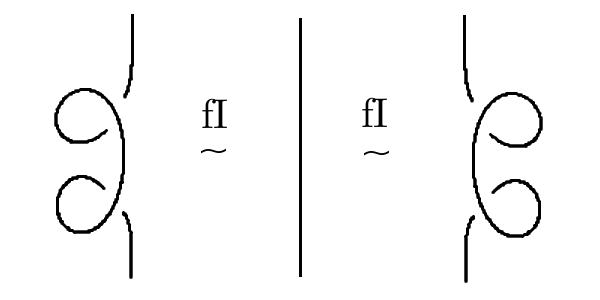} \quad \includegraphics{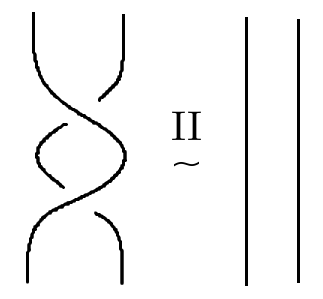}
\quad \includegraphics{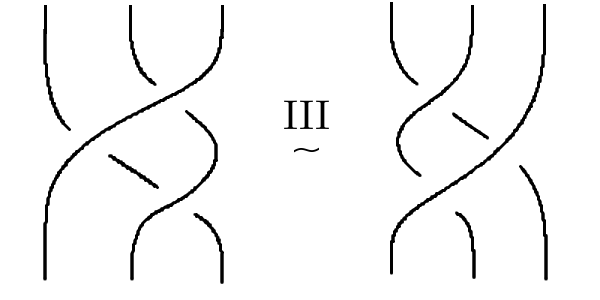}
\]
An \textit{oriented framed link} has a choice of orientation for each
component of the link, and oriented framed Reidemeister moves respect
orientation. For each component of a link, the \textit{framing number}
or \textit{writhe} of the component is the sum of crossing signs 
\[\includegraphics{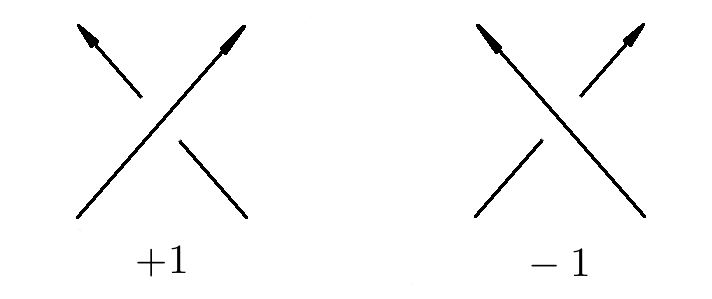}\]
at each crossing where both strands are from the component in question.
Note that framed Reidemeister moves preserve the framing numbers of each
component. An \textit{unframed link} or just a \textit{link} is an 
equivalence class of link diagrams under the equivalence relation obtained 
by replacing the framed type I move with the unframed type I move:
\[\includegraphics{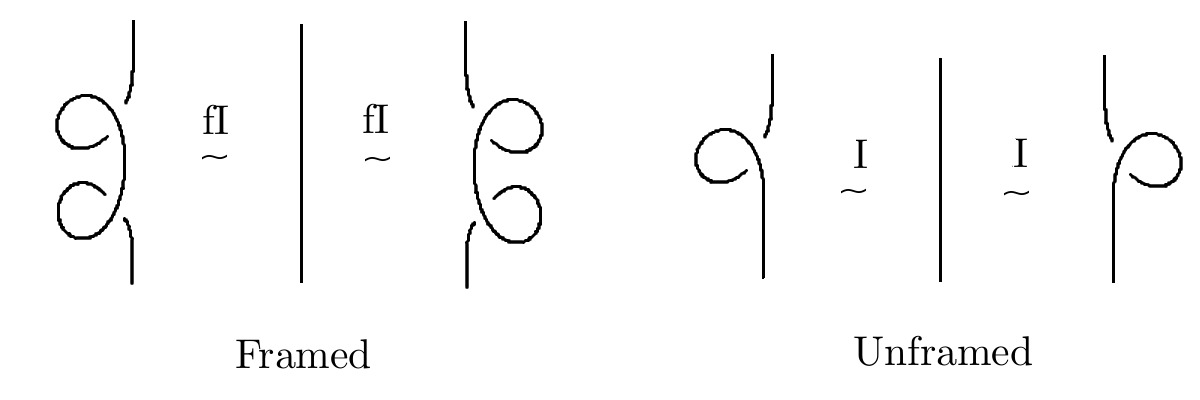}\]
A link diagram represents a union of disjoint simple closed curves in 
$\mathbb{R}^3$; each simple closed curve is a \textit{component} of the link.
A link with a single component is a \textit{knot}.

Let $X$ be a set. We would like to define an algebraic structure on 
$X$ such that \textit{labelings}, i.e. assignments of elements of $X$ to 
\textit{semiarcs} (portions of the diagram between adjacent over or 
under crossing points) in an oriented blackboard-framed link diagram $L$, 
are preserved under framed oriented
Reidemeister moves. To define such an algebraic structure, we can think
of crossings in a link diagram as determining a map 
$B:X\times X\to X\times X$ as pictured.
\[\includegraphics{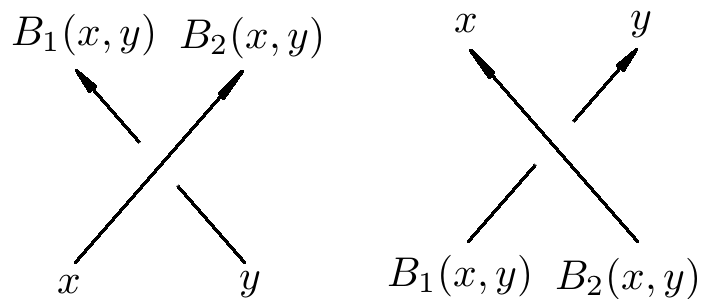}\]
Translating the oriented blackboard framed Reidemeister moves into
conditions on $B$, we obtain the following definition (see also 
\cite{FRS,KR,FJK,N2}).

\begin{definition}\textup{
Let $X$ be a set and let $\Delta:X\to X\times X$ be the diagonal map
$\Delta(x)=(x,x).$ An invertible map $B:X\times X\to X\times X$ is a 
\textit{birack map} if 
\begin{itemize}
\item[(i)] There exists a unique invertible \textit{sideways map}
$S:X\times X\to X\times X$ satisfying
\[S(B_1(x,y),x)=(B_2(x,y),y) \]
\item[(ii)] The components of the composition of the diagonal map
with the sideways map and with its inverse, $(S^{\pm 1}\Delta)_{1,2}$,
are bijections, and
\item[(iii)] $B$ satisfies the \textit{set-theoretic Yang-Baxter equation},\
\[(B\times I)(I\times B)(B\times I)=(I\times B)(B\times I)(I\times B).\]
\end{itemize}
}\end{definition}

Invertibility of $B$ and axiom (i) guarantee that labelings before and after
Reidemeister type II moves correspond bijectively.
\[\includegraphics{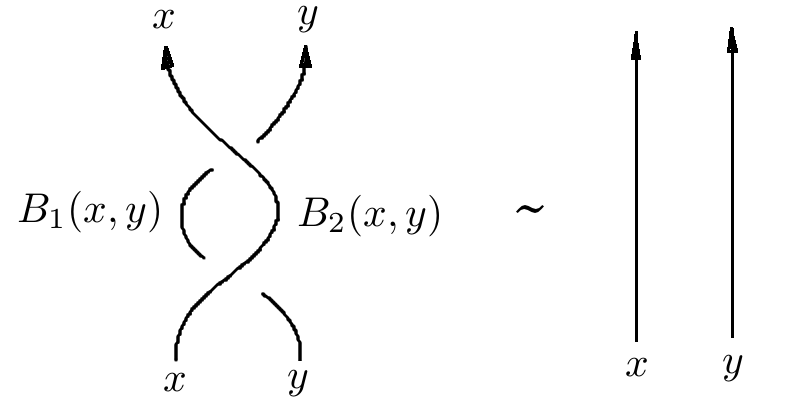}\quad \includegraphics{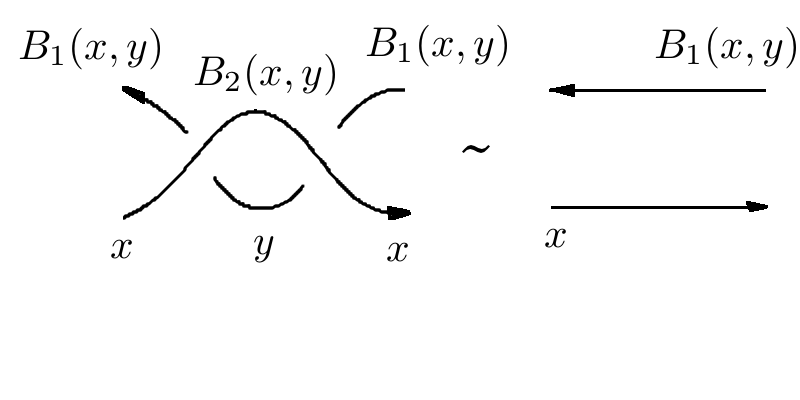}\]
These conditions can be summarized with the ``adjacent labels rule'', which
says any two adjacent labels at a crossing determine the other two labels.

Axiom (ii) guarantees that the label on the input semiarc of a kink
determines the other labels; in particular, the map taking the input
label to the output label at a positive kink, 
$\pi=(S\Delta)_1(S\Delta)_2^{-1}$, is a bijection called the \textit{kink map}.
\[\includegraphics{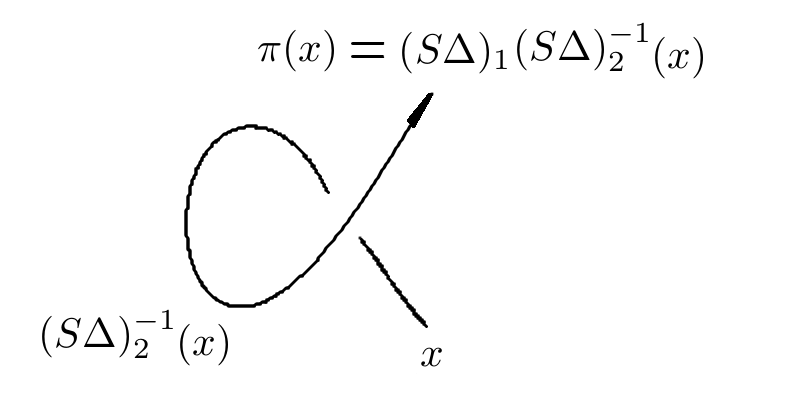}\]
This is enough to guarantee that labelings of diagrams before and after 
framed type I moves correspond bijectively. See \cite{N2} for more.
\[\includegraphics{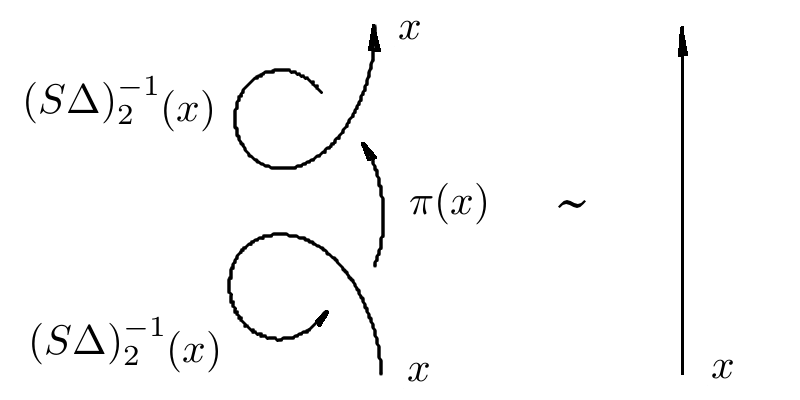}\]

Axiom (iii) guarantees that diagrams before and after Reidemeister type III
moves correspond bijectively. Note that horizontal stacking here corresponds 
to Cartesian product $\times$ and vertical stacking corresponds to function
composition.
\[(B\times I)(I\times B)(B\times I)\quad
\raisebox{-0.65in}{\includegraphics{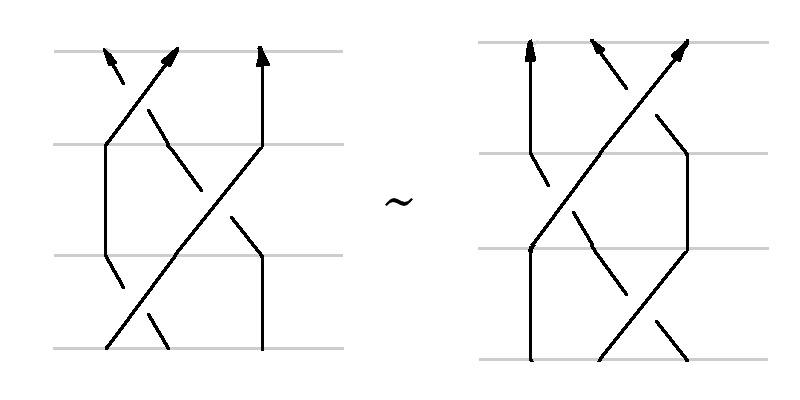}} \quad 
(I\times B)(B\times I)(I\times B)
\]

If $X$ is a finite set, then the kink map $\pi$ is an element of the symmetric
group on the elements of $X$. In particular, the \textit{birack rank} $N$
of $X$ is the exponent of $\pi$, i.e. the smallest positive integer $N$
such that $\pi^N$ is the identity map on $X$. Two link diagrams which are
related by framed oriented Reidemeister moves and \textit{$N$-phone cord moves}
have birack labelings by $X$ which are in one-to-one correspondence.
\[\includegraphics{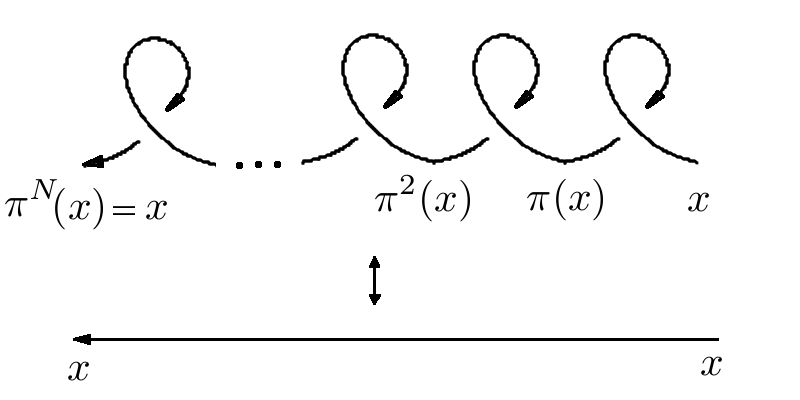}.\]
A birack of rank $N=1$ is a \textit{strong biquandle}.

Examples of birack structures include:
\begin{itemize}
\item \textit{Constant Action Biracks}. Let $X$ be a set and 
$\sigma,\tau:X\to X$ bijections such that $\sigma\tau=\tau\sigma$. Then
$B(x,y)=(\sigma(y),\tau(x))$ defines a birack map on $X$ with kink map
$\pi=\sigma\tau^{-1}$.
\item \textit{$(t,s,r)$-Biracks}. Let $X$ be a module over the ring
$\tilde\Lambda=\mathbb{Z}[t^{\pm 1},s,r^{\pm 1}]/(s^2-(1-tr)s)$. 
Then $B(x,y)=(sx+ty,rx)$ is a birack map on $X$ with kink map
$\pi(x)=(tr+s)x$.
\item \textit{Fundamental Birack of an oriented framed link}.
Given an oriented framed link diagram $L$, let $Y$ be a set of
generators corresponding to semiarcs in $L$. Then the set of 
\textit{birack words} determined by $L$ includes elements of $Y$
and expressions of the form $B_{1,2}^{\pm 1}(x,y)$ and $S_{1,2}^{\pm}(x,y)$
where $x,y$ are birack words in $L$. Then the \textit{fundamental birack}
of $L$, denoted $BR(L)$, is the set of equivalence classes of birack words
under the equivalence relation determined by the birack axioms and the
crossing relations in $L$. See 
\cite{FJK} or \cite{KR} for more.
\end{itemize}

As in other categories, we have the standard notions of homomorphisms and
sub-objects. More precisely, let $X,Y$ be sets with birack maps $B,B'$
and let $Z\subset X$. Then
\begin{itemize}
\item A \textit{homomorphism of biracks} is a map $f:X\to Y$ such that
\[B'(f\times f)=(f\times f)B,\] and
\item $Z$ is a \textit{subbirack} of $X$ if the restriction $B_{Z\times Z}$
of $B$ to $Z\times Z\subset X\times X$ is a birack map.
\end{itemize}

If $X=\{x_1,\dots,x_n\}$ is a finite birack, we can specify a birack 
structure on $X$ with a pair of operation matrices expressing the maps
$B_{1}(x,y)$ and $B_{2}(x,y)$ as binary operations. More precisely, a 
\textit{birack matrix} $[M|M']$ has two $n\times n$ block matrices $M$, 
$M'$ such that 
\[M_{i,j}=k \quad \mathrm{and} \quad M'_{i,j}=l\]
where $x_k=B_1(x_j,x_i)$ and $x_l=B_2(x_i,x_j)$. Note the reversed order
of the input components of $M$; the notation is chosen so that the row
number and output are on the same strand.

\begin{example}\label{ex1}
\textup{Consider the $(t,s,r)$-birack structure on $X=\mathbb{Z}_4$ 
given by
$B(x,y)=(2x+3y,3x)$. $X$ has birack matrix 
\[M_X=\left[\begin{array}{rrrr|rrrr}
1 & 3 & 1 & 3 & 1 & 1 & 1 & 1 \\
4 & 2 & 4 & 2 & 4 & 4 & 4 & 4 \\
3 & 1 & 3 & 1 & 3 & 3 & 3 & 3 \\
2 & 4 & 2 & 4 & 2 & 2 & 2 & 2 \\
\end{array}\right]\] 
where $x_1=0, x_2=1, x_3=2$ and $x_4=3\in \mathbb{Z}_4$.}
\end{example}

If $L$ is an oriented framed link and $X$ is a finite birack, then
a homomorphism $f:BR(L)\to X$ assigns an element $f(g)$ of $X$ to each
generator $g$ of $BR(L)$, so such a homomorphism determines a labeling
of the semiarcs of $L$ with elements of $X$. Conversely, such a labeling 
defines a homomorphism if and only if the crossing relations are
satisfied at every crossing. In particular, the set $\mathrm{Hom}(BR(L),X)$
is a finite set; its cardinality $|\mathrm{Hom}(BR(L),X)|=\Phi^B_X(L)$
is a computable invariant of framed oriented links known as the 
\textit{basic birack counting invariant.}

Each component of a  $c$-component link can have any integer as its
framing number; thus, for any $c$-component link, there is a 
$\mathbb{Z}^c$-lattice of framed links and a corresponding 
$\mathbb{Z}^c$-lattice of basic counting invariant values $\Phi^B_X(L)$. 
If a birack $X$ has rank $N$ and
$L$ and $L'$ are related by $N$-phone cord moves, then every $X$-labeling
of $L$ corresponds to a unique $X$-labeling of $L'$ and vice-versa;
thus the $\mathbb{Z}^c$-lattice of basic counting invariant values is tiled
by a tile of side length $N$. Summing the numbers of birack labelings over 
a complete tile of framing vectors mod $N$ then yields an invariant of
unframed links known as the \textit{integral birack counting invariant},
\[\Phi_X^{\mathbb{Z}}(L)=\sum_{\vec{w}\in (\mathbb{Z}_N)^c} 
|\mathrm{Hom}(BR(L,\vec{w}),X)|\]
where $(L,\vec{w})$ is a diagram of $L$ with framing vector $\vec{w}$.

\begin{example}\label{ex1.5}\textup{
Let $X$ be the birack in example \ref{ex1}, i.e. the $(t,s,r)$-birack on
$\mathbb{Z}_4$ with $t=r=3$ and $s=2$. 
We have $(tr+s)=(3)(3)+2=3$ and $3^2=1$ in $\mathbb{Z}_4$, so $X$ has birack 
rank $N=2$. To compute the counting invariant for a link $L$, then, we need 
to count $X$-labelings of a set of diagrams of $L$ with every 
combination of even and odd writhes on the components of $L$. For example, 
the link $L4a1$ has a total of $\Phi_X^{\mathbb{Z}}(L4a1)=36$ labelings
by $X$ over a complete tile of framings mod 2, while the Hopf link $L2a1$
has a total of $\Phi_X^{\mathbb{Z}}(L2a1)=20$.
\[\begin{array}{cc}
\includegraphics{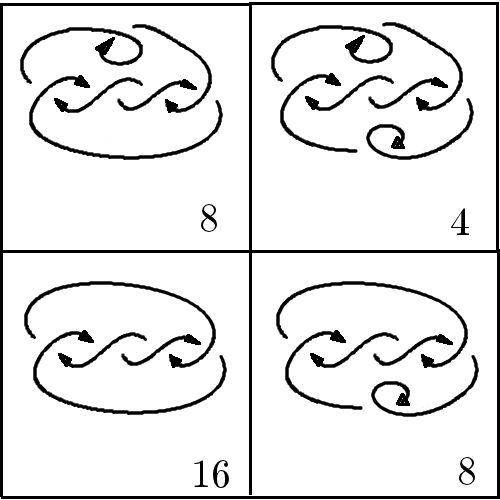} & \includegraphics{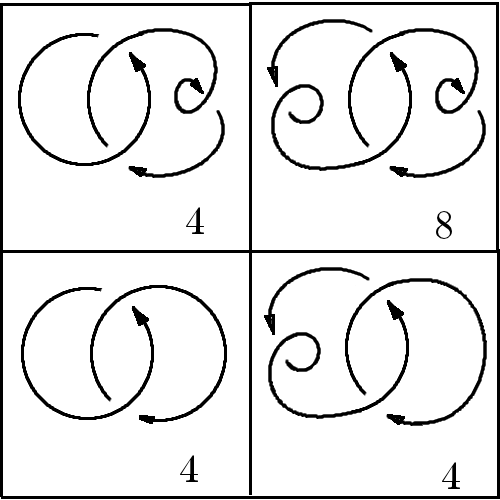} \\
L4a1 & L2a1 \\
\end{array}\]
}
\end{example}

For a given finite birack $X$ and link $L$, if $\Phi^{Z}$
If two oriented links have the same value of $\Phi_X^{\mathbb{Z}}$, it could 
be that the links are the same, or it might be a coincidence. An
\textit{enhancement} of the counting invariant is a stronger invariant which
determines $\Phi_X^{\mathbb{Z}}$ but contains additional information about 
$L$ which can distinguish links which coincidentally have the same
$\Phi_X^{\mathbb{Z}}$ value.

\begin{example}\textup{
In \cite{N2}, an enhancement is defined by keeping track of which framings
contribute which labelings to $\Phi_X^{\mathbb{Z}}$. More precisely, let us
abbreviate $q_1^{w_1}q_2^{w_2}\dots q_c^{w_c}$ as $q^{(w_1,w_2,\dots,w_c)}$. Then
the \textit{writhe enhanced counting invariant} is 
\[\Phi_X^{W}(L)=\sum_{\vec{w}\in (\mathbb{Z}_N)^c} 
|\mathrm{Hom}(FB(L,\vec{w}),X)|q^{\vec{w}}.\]
Then for the birack $X$ in example \ref{ex1}, we have 
$\Phi_X^W(L4a1)=16+8q_1+8q_2+4q_1q_2$ and
$\Phi_X^W(L2a1)=4+4q_1+4q_2+8q_1q_2$. Note that $\Phi_X^W$ evaluated at
$u=1$ yields $\Phi_X^{\mathbb{Z}}$. 
}\end{example}

\section{\large\textbf{Birack Dynamical Cocycles and Birack Homomorphisms}}\label{BP}

In this section we define birack dynamical cocycles and 
introduce a new enhancement of the birack counting invariant.

\begin{definition}\textup{
Let $X$ be a birack of rank $N$, $S$ a set with identity map $I$, 
and consider a set $D$ of maps $D_{x,y}:S\times S\to S\times S$. 
Such a collection of maps defines a \textit{birack dynamical cocycle} if
\begin{list}{}{}
\item[(i)] Every $D_{x,y}$ is invertible,
\item[(ii)] For each $D_{x,y}$ there is a unique invertible
map $S_{x,y}:S\times S\to S\times S$ such that for all $a,b\in S$, we have
\[S((D_{x,y})_1(a,b),a)=((D_{x,y})_2(a,b),b),\]
\item[(iii)] The maps $(S_{x,y}^{\pm 1}\Delta)_{1,2}:S\to S$ are bijections
\item[(iv)] For every  $x,y,z\in X$, the \textit{$X$-labeled Yang Baxter 
equations} 
\[(I\times D_{B_2(x,B_1(y,z)),B_2(y,z)})(D_{x,B_1(y,z)}\times I)(I\times D_{y,z})
=(D_{B_1(x,y),B_1(B_2(x,y),z)}\times I)(I\times D_{B_2(x,y),z})(D_{x,y}\times I)
\]
are satisfied, and
\item[(v)]  For every $x\in X$, we have
\[\pi_{\alpha(\pi^Nx),\pi^N x}\dots\pi_{\alpha(\pi x),\pi x} \pi_{\alpha(x),x} =I\]
where $\alpha_{x,y}=(S_{x,y}\Delta)_2^{-1}$ and 
$\pi_{x,y}=(S_{x,y}\Delta)_1\alpha_{x,y}$ 
\end{list}
}\end{definition}

The birack dynamical cocycle axioms come from the $X$-labeled framed 
oriented Reidemeister moves and the $N$-phone cord move
where we think of the elements of $S$ as ``beads'' on each semiarc.
The operation $D_{x,y}$ is then the result of pushing the beads through a 
crossing with input birack labels $x,y$:
\[\raisebox{-0.5in}{\includegraphics{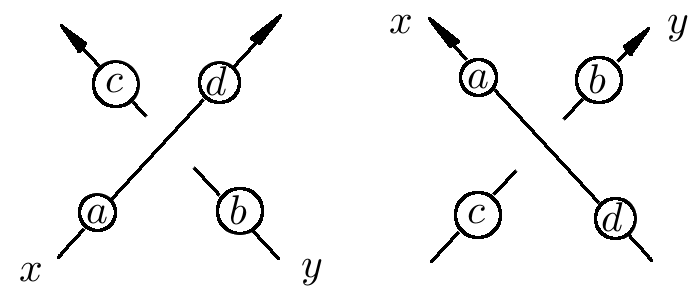}} \quad D_{x,y}(a,b)=(c,d)\]
For a fixed $X$-labeling $f$ of an oriented link diagram $L$, let
$\mathcal{L}_S(f)$ be the number of assignments of elements of $S$ to semiarcs
in $L$ such that the above pictured condition is satisfied at every crossing.
The birack dynamical cocycle axioms are chosen so that for every $S$-labeling
of an $X$-labeled oriented link diagram before a Reidemeister or $N$-phone cord move, there is a unique corresponding $S$-labeling after the move. That is, 
$|\mathcal{L}_S(f)|$ is an invariant of $X$-labeled oriented framed isotopy mod $N$.

Analogously to \cite{CNS}, we define the \textit{birack dynamical cocycle 
invariant} by counting the bead labelings as a signature for each birack 
labeling:

\begin{definition}
\textup{Let $X$ be a finite birack of rank $N$, $D$ a birack dynamical 
cocycle and $L$ an oriented link of $c$ components. The \textit{birack 
dynamical cocycle enhanced multiset} is the multiset
\[\Phi_X^{D,M}(L)=\{|\mathcal{L}_S(f)| \ :\ f\in\mathrm{Hom}(BR(L,\vec{w}),X),
\vec{w}\in(\mathbb{Z}_N)^c\}\]
and the \textit{birack dynamical cocycle enhanced polynomial} is
\[\Phi_X^{D}(L)=\sum_{\vec{w}\in(\mathbb{Z}_N)^c}\left(\sum_{
f\in\mathrm{Hom}(BR(L,\vec{w}),X)}
 u^{|\mathcal{L}_S(f)|}\right)\]
where $(L,\vec{w})$ is a diagram of $L$ with writhe vector $\vec{w}$.
}
\end{definition}

By construction, we have

\begin{theorem}
If $L$ and $L'$ are ambient isotopic oriented links, then
$\Phi_X^{D,M}(L)=\Phi_X^{D,M}(L')$ and 
$\Phi_X^{D}(L)=\Phi_X^{D}(L')$.
\end{theorem}

We can simplify the new enhancement with the observation that an $S$-labeling
of an $X$-labeled diagram is really a labeling by pairs in $X\times S$, and the
birack dynamical cocycle axioms are precisely the conditions required to make
$X\times S$ a birack under the map $B\rtimes D$ defined by 
\[B\rtimes D((x,a),(y,b))=((B_1(x,y),(D_{x,y})_1(a,b)),(B_2(x,y),(D_{x,y})_2(a,b))).\]
We will denote this birack structure on $X\times S$ as $X\rtimes_D S$.

\begin{example}\textup{
If $X$ and $S$ are biracks with birack maps $B$ and $C$ respectively, then 
the Cartesian product $X\times S$ has a birack map $B\times C$:
\[(B\times C)((a,x),(b,y))=((B_1(a,b),C_1(x,y)),(B_2(a,b),C_2(x,y)))\]
and the dynamical cocycle maps are given by $D_{x,y}=C$ for all $x,y\in X$. 
We will denote this birack structure simply by $X\times S$. In particular,
we can think of a dynamical cocycle as generalization of the Cartesian product
structure where the map on the $S$ components depends on the $X$ components.
}\end{example}

\begin{example}\textup{
If $X$ is a birack and $S$ has an $X$-module structure over a ring
$R$ given by a matrix $[T|S|R]$ (see \cite{BN}), then $X\times S$
has birack dynamical cocycle given by 
\[D_{x,y}(a,b)=(s_{x,y}a+t_{x,y}b,r_{x,y}a).\]
}\end{example}

In particular, the ``forgetful homomorphism'' $p:X\times S\to X$ defined
by $p(x,s)=x$ is a birack homomorphism which we may think of as a
coordinate projection map. We can thus think of the enhancement
$\Phi_{X}^D$ as starting with $X\times S$-labelings of $L$ and collecting
together the $X\times S$ labelings which project to the same $X$-labeling.
This leads us to a generalization: let $p:X\to Y$ be any birack homomorphism.
For each $Y$-labeling $f\in\mathrm{Hom}(BR(L),Y)$, we obtain a signature 
$\sigma(f)=|\{g:BR(L)\to X\ : \ pg=f\}|$ for $f$ by counting the number of 
birack homomorphisms $g:BR(L)\to X$ such that the diagram commutes.
\[\includegraphics{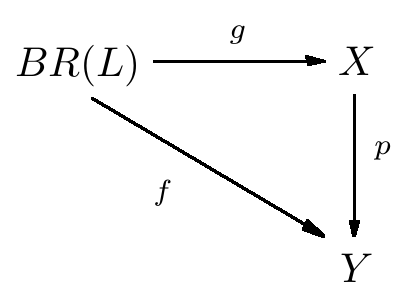}\]
The multiset of such signatures is an enhancement of $\Phi_Y^{\mathbb{Z}}$.
Note that not every labeling of $L$ by $Y$ necessarily factors through $p$; 
some $\sigma(f)$s could be zero. Such labelings contribute $u^0=1$ to
$\Phi_p(L)$, so the constant term in $\Phi_p$ counts the number of 
$Y$-labelings of $L$ which do not factor through $p$.

\begin{definition}
\textup{Let $X$ be a finite birack of rank $N$, $p:X\to Y$ a birack 
homomorphism and $L$ an oriented link of $c$ components. For each
$f\in\mathrm{Hom}(BR(L,\vec{w}),Y)$, let 
\[\sigma(f)=|\{g\in\mathrm{Hom}(BR(L,\vec{w}),X))\ |\ pg =f\}|,\] the number
of $X$-labelings of $(L,\vec{w})$ that project to $f$. 
The \textit{birack homomorphism enhanced multiset} is the multiset
\[\Phi_p^{M}(L)=\{ \sigma(f)\ |\
f\in\mathrm{Hom}(BR(L,\vec{w}),Y), \vec{w}\in(\mathbb{Z}_N)^c\}\]
and the \textit{birack homomorphism enhanced polynomial} is
\[\Phi_p(L)=
\sum_{\vec{w}\in(\mathbb{Z}_N)^c}
\left(\sum_{f\in\mathrm{Hom}(FB(L,\vec{w},Y))} 
u^{\sigma(f)}\right).\]
}\end{definition}

If $X=Y\rtimes_D Z$ and $p:X\to Y$ is projection onto the first
factor, then $\phi_Y^{D,M}=\Phi_p^{M}$ and $\phi_Y^{D}=\Phi_p$.
In the case of the Cartesian product of two biracks $X=Y\times Z$, we can 
say what the $\Phi_p$ looks like:
\begin{proposition}
If $X=Y\times Z$ is a birack of rank $N$ and $p:X\to Y$ is the coordinate
projection homomorphism, then we have 
\[\Phi_p(L)=\sum_{\vec{w}\in (\mathbb{Z}_N)^c} 
\Phi_Y^B(L,\vec{w})u^{\Phi_Z^B(L,\vec{w})}.\]
\end{proposition}

\begin{proof}
We simply note that for each writhe vector $\vec{w}\in (\mathbb{Z}_N)^c$, 
the $Y$- and $Z$-labelings are independent. Hence, for each $Y$-labeling
of a diagram $L$ with framing vector $\vec{w}$, there are $\Phi_Z^B(L)$
$Z$-labelings.
\end{proof}



Next, a few straightforward observations.

\begin{proposition}
If $p:X\to Y$ is an isomorphism, then $\Phi_p(L)=\Phi_X^{\mathbb{Z}}(L)u$.
\end{proposition}

\begin{proposition}
If $p:X\to Y$ is a constant map, then $\Phi_p(L)=u^{\Phi_X^{\mathbb{Z}}(L)}$.
\end{proposition}

\begin{remark}\textup{
As with many combinatorially-defined link invariants, $\Phi_p$ extends
to virtual knots and links by ignoring the virtual crossings, i.e. by
not dividing semiarcs at crossings. See \cite{K}
for more about virtual knots and links.}
\end{remark}

We end this section with a connection to recent work on $(t,s)$-racks.
Let $\ddot\Lambda=\mathbb{Z}[t^{\pm 1},s]/(s^2-(1-t)s)$.
A \textit{$(t,s)$-rack} is a birack structure on a $\ddot\Lambda$-module 
$X$ given by
\[B(x,y)=(ty+sx,x).\]
In \cite{CN}, an invariant $\Phi^s_{X}$ was defined by 
collecting together the $X$-labelings of a link diagram which project
to the same $sX$ labeling under the map $s:X\to sX$.
We note that $\Phi_X^s$ is the same as $\Phi_s$ in our present terminology.

\section{\large\textbf{Computations and Applications}}\label{CA}

In this section we collect a few computations, examples and applications 
of the new invariant. We begin with an explicit example of computing
$\Phi_p$.

\begin{example}\textup{
Let $X$ and $Y$ be the biracks with matrices
\[M_X=\left[\begin{array}{rrr|rrr}
2 & 2 & 1 & 1 & 1 & 1 \\
1 & 1 & 2 & 2 & 2 & 2 \\
3 & 3 & 3 & 3 & 3 & 3 \\
\end{array}\right]\quad \mathrm{and} \quad 
M_Y=\left[\begin{array}{rr|rr}
a & a & a & a \\
b & b & b & b \\
\end{array}\right]\]
and let $p:X\to Y$ be given by $p(1)=p(2)=a, p(3)=b$.}

\textup{
The kink map for $X$ is the permutation $\pi=(12)$, so $X$ has birack rank 
$N=2$; hence we must find all labelings of the semiarcs in diagrams of 
$L$ with writhe vectors $(0,0)$, $(0,1)$, $(1,0)$ and $(1,1)$ mod 2 which 
satisfy the labeling condition
\[\includegraphics{sn-ew-1.png}.\]
One can do this by choosing diagrams with the required framings mod 2 and
simply listing all possible assignments of elements of
$X$ to semiarcs in $L$, keeping only those which satisfy the condition;
our code uses an algorithm which propagates labels through partially-labeled
diagrams. Our \texttt{python} code is available at 
\texttt{www.esotericka.org}.}

\textup{ Let us compute $\Phi_p$ for the Hopf link $L2a1$. There are 16 
$X$-labelings of $L2a1$ over a tile of framings mod $2$, as depicted.
\[\includegraphics{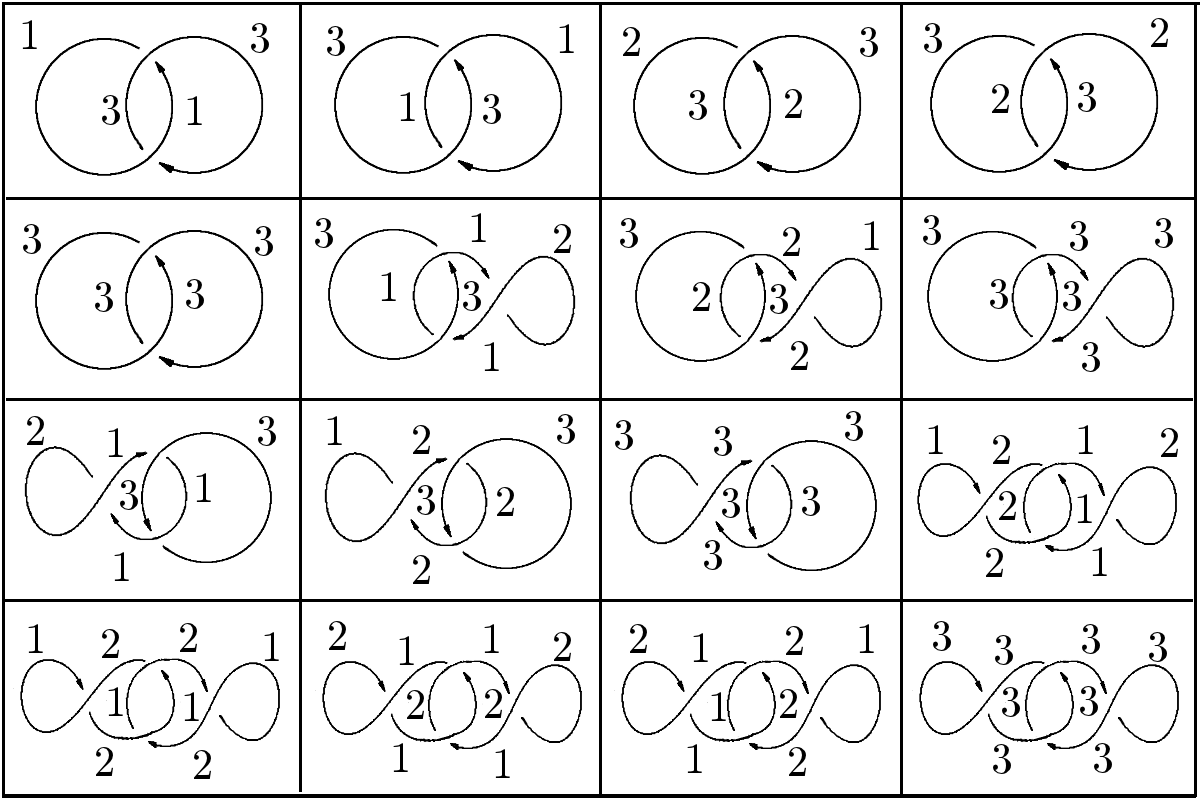}\]
These project to the pictured $Y$-labelings.
\[\includegraphics{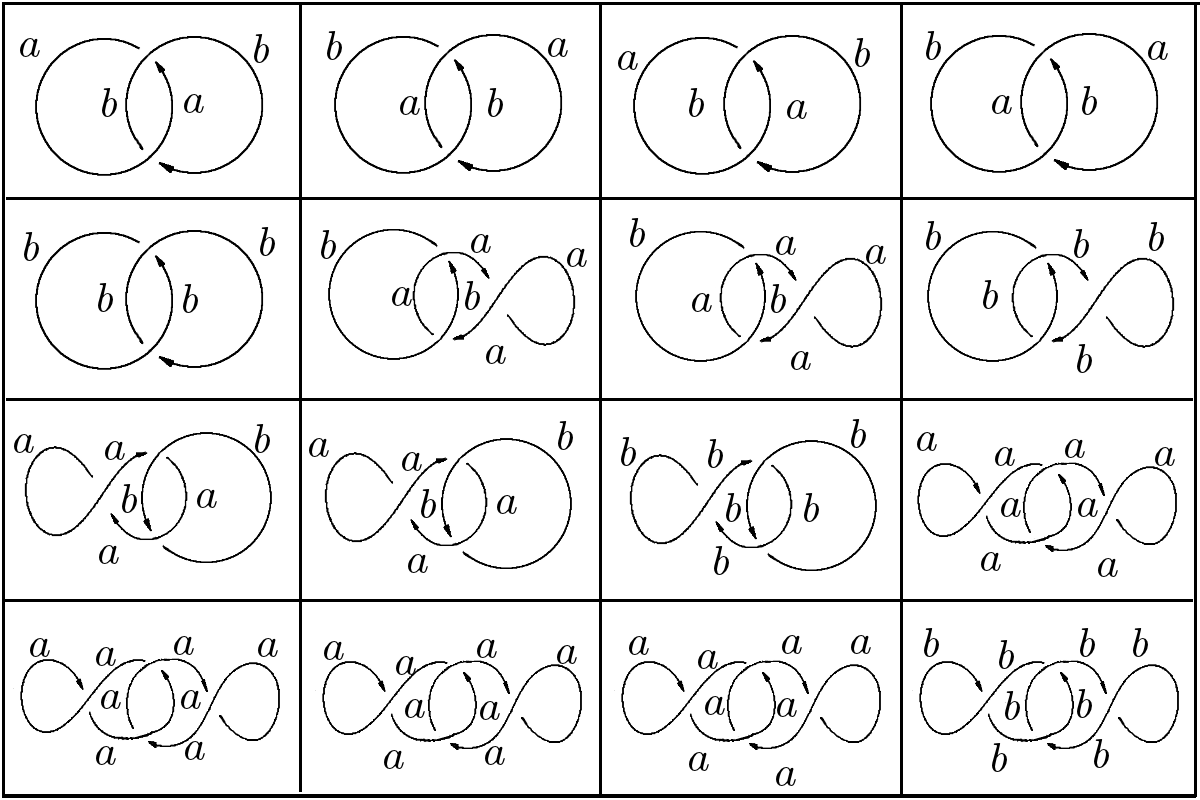}\]
with contributions $u+2u^2$ from the $(0,0)$-framing, $u+u^2$ from the
$(0,1)$ framing,  $u+u^2$ from the $(1,0)$ framing and $u+u^4$ from the 
$(1,1)$ framing to yield $\Phi_p(L2a1)=4u+4u^2+u^4.$
}\end{example}

Let $p:X\to Y$ be a birack projection. The invariant $\Phi_p$ can be 
understood as an enhancement of the birack counting invariant with respect
to $X$, thinking of collecting together $X$-labelings which project to the
same $Y$-labeling; alternatively, we can understand $\Phi_p$ as an enhancement
of the counting invariant with respect to $Y$, where for each $Y$-labeling
we find how many $X$-labelings project to it via $p$. The next examples  
show that $\Phi_p$ is a proper enhancement, i.e. $\Phi_p$ is not determined 
by either $\Phi_X^{\mathbb{Z}}$ or $\Phi_Y^{\mathbb{Z}}$.

\begin{example}\textup{The biracks $X$ and $Y$ with listed matrices have
projection maps including $p:X\to Y$ below:
\[M_X=\left[\begin{array}{rrrrrr|rrrrrr}
2 & 2 & 2 & 1 & 1 & 1 & 2 & 1 & 3 & 1 & 1 & 1 \\ 
3 & 3 & 3 & 2 & 2 & 2 & 1 & 3 & 2 & 2 & 2 & 2 \\
1 & 1 & 1 & 3 & 3 & 3 & 3 & 2 & 1 & 3 & 3 & 3 \\
5 & 5 & 5 & 5 & 5 & 5 & 4 & 4 & 4 & 4 & 4 & 4 \\
6 & 6 & 6 & 6 & 6 & 6 & 5 & 5 & 5 & 5 & 5 & 5 \\
4 & 4 & 4 & 4 & 4 & 4 & 6 & 6 & 6 & 6 & 6 & 6 \\
\end{array}\right],\quad 
M_Y=\left[\begin{array}{rr|rr}
a & a & a & a \\
b & b & b & b \\
\end{array}\right]
\]
\[p(1)=p(2)=p(3)=a, \ p(4)=p(5)=p(6)=b.\]
The trefoil $3_1$ and the figure eight $4_1$ are not distinguished by
the counting invariant $\Phi_Y^{\mathbb{Z}}(3_1)=\Phi_Y^{\mathbb{Z}}(4_1)=2$, 
but the enhanced invariant
$\Phi_p(3_1)=u^3+3u^9\ne \Phi_p(4_1)=4u^3$ detects the difference.
}\end{example}

\begin{example}\textup{The biracks $X$ and $Y$ with listed matrices have
projection maps including $p:X\to Y$ below:
\[M_X=\left[\begin{array}{rrrrr|rrrrr}
2 & 5 & 4 & 2 & 5 & 5 & 5 & 4 & 5 & 5  \\ 
1 & 4 & 5 & 1 & 4 & 1 & 1 & 5 & 1 & 1  \\
3 & 3 & 3 & 3 & 3 & 3 & 3 & 3 & 3 & 3  \\
5 & 2 & 1 & 5 & 2 & 2 & 2 & 1 & 2 & 2  \\
4 & 1 & 2 & 4 & 1 & 4 & 4 & 2 & 4 & 4  \\
\end{array}\right],\quad 
M_Y=\left[\begin{array}{rr|rr}
a & a & a & a \\
b & b & b & b \\
\end{array}\right]
\]
\[p(1)=p(2)=p(4)=p(5)=a,\ p(3)=b.\]
The Hopf link $L2a1$ and the $(4,2)$-torus link $L4a1$ are not 
distinguished by the counting invariant $\Phi_X^{\mathbb{Z}}(L2a1)
=\Phi_X^{\mathbb{Z}}(L4a1)=20$, but the enhanced invariant
$\Phi_p(L2a1)=4u+4u^4+2u^8\ne \Phi_p(L4a1)=4u+4u^4+u^{16}$ detects 
the difference.
}\end{example}

In our final example we compare $\Phi_p$ values on certain virtual knots
to demonstrate that $\Phi_p$ is not determined by the Jones polynomial
or the generalized Alexander polynomial.

\begin{example}\textup{
The virtual knot $3.7$ has Jones polynomial $J(3.7)=1$, the same
as the unknot; $3.7$ also has generalized Alexander polynomial
$(s - 1)(s + 1)(t - 1)(t + 1)(st - 1)$, the same as the virtual knot
$4.47$. However, the biracks $X,Y$ with homomorphism $p$
\[M_x=\left[\begin{array}{cccccc|cccccc}
3 & 2 & 1 & 1 & 1 & 1 & 2 & 2 & 2 & 1 & 1 & 1 \\
1 & 3 & 2 & 2 & 2 & 2 & 1 & 1 & 1 & 2 & 2 & 2 \\
2 & 1 & 3 & 3 & 3 & 3 & 3 & 3 & 3 & 3 & 3 & 3 \\
6 & 4 & 5 & 5 & 5 & 5 & 5 & 4 & 6 & 4 & 4 & 4 \\
4 & 5 & 6 & 6 & 6 & 6 & 6 & 5 & 4 & 5 & 5 & 5 \\
5 & 6 & 4 & 4 & 4 & 4 & 4 & 6 & 5 & 6 & 6 & 6
\end{array}\right],
\quad M_y=\left[\begin{array}{cc|cc}
a & a & a & a \\
b & b & b & b \\
\end{array}
\right]\]
\[ p(1)=p(2)=p(3)=a, p(4)=p(5)=p(6)=b.\] 
distinguish $3.7$ from the unknot and from $4.47$ with 
$\Phi_p(3.7)=3u^9+u^3$ while $\Phi_p(\mathrm{Unknot})=\Phi_p(4.47)=4u^3$. Hence,
$\Phi_p$ can distinguish knots with the same Jones and generalized Alexander
polynomials.
}\end{example}

\section{\large\textbf{Questions}}\label{Q}

We end with a few open questions for future research.

What is the relationship between birack projection invariants and birack 
cocycle invariants? Indeed, what is the role of birack dynamical cocycles 
in birack homology and cohomology? 

What structures are analogous to birack dynamical cocycles in the 
settings of virtual biracks and twisted virtual biracks? What happens when 
we add a shadow structure as in \cite{NP}?

We have used primarily small cardinality examples for speed of computation
and convenience of presentation; 
we note that even these small cardinality examples with $Y$ the trivial
birack on two elements suffice to show that $\Phi_p$ is not determined
by the integral counting invariant, the Jones  polynomial or the generalized
Alexander polynomial. Faster algorithms for computation of
$\Phi_p$ for larger biracks should allow more exploration of $\Phi_p$.

\bigskip

\noindent
\textsc{Department of Mathematical Sciences \\
Claremont McKenna College \\
850 Columbia Ave. \\
Claremont, CA 91711}


\begin{thebibliography}{10}

\bibitem{AG}{N. Andruskiewitsch and M. Gra\~{n}a.
From racks to pointed Hopf algebras. 
\textit{Adv. Math.} \textbf{178} (2003) 177-243.}


\bibitem{KA}{D. Bar-Natan (Ed.). The Knot Atlas.
\texttt{http://katlas.math.toronto.edu/wiki/Main\underline{\ }Page}.}

\bibitem{BN}{R. Bauernschmidt and S. Nelson. Birack modules and their 
link invariants. arXiv:1103.0301}

\bibitem{CN}{J. Ceniceros and S. Nelson. (t,s)-racks and their link invariants.
\textit{Int'l J. Math.} \textbf{23} (2012) 1250001-1--1250001-19.}

\bibitem{CNS}{A. Crans, S. Nelson and A. Sarkar. Enhancements of the rack counting invariant via N-reduced dynamical cocycles. arXiv:1108.4387; to appear in
\textit{New York J. Math.} }

\bibitem{FJK}{R. Fenn, M. Jordan-Santana and L. Kauffman. Biquandles 
and virtual links.  \textit{Topology Appl.}  \textbf{145}  (2004) 157-175.}

\bibitem{FRS}{R. Fenn, C. Rourke and B. Sanderson. 
Trunks and classifying spaces. \textit{Appl. Categ. Structures} \textbf{3} 
(1995) 321--356.}

\bibitem{K}{L. Kauffman. Virtual Knot Theory. \textit{European J. Combin.}
\textbf{20} (1999) 663-690.}

\bibitem{KR}{L. H. Kauffman and D. Radford. Bi-oriented quantum algebras, 
and a generalized Alexander polynomial for virtual links. 
\textit{Contemp. Math}. \textbf{318} (2003) 113-140.}

\bibitem{N2}{S. Nelson. Link invariants from finite biracks.$B_{1,2}(x,y)$
arXiv:1002.3842.}

\bibitem{NP}{S. Nelson and K. Pelland. Birack shadow modules and their link 
invariants. arXiv:1106.0336}


\end{thebibliography}
\end{document}